\documentclass[english,11pt,a4paper,leqno]{amsart}
\usepackage{amsmath,amstext, amsthm, amssymb}
\usepackage{appendix}

\usepackage{color}
\definecolor{liens}{rgb}{1,0,0}
\usepackage[colorlinks=true, linkcolor=blue, 
hyperfootnotes=true,citecolor=blue,urlcolor=black]{hyperref}

\usepackage{latexsym, amsthm, amsfonts, amsmath,amssymb,graphicx,xcolor}
\input xy
\xyoption{all}
\usepackage[latin1]{inputenc} 
\usepackage[T1]{fontenc} 
\usepackage[english]{babel} 

\newtheorem{theo}{Theorem}

\newtheorem{lem}[theo]{Lemma}
\newtheorem{propo}[theo]{Proposition}
\newtheorem{coro}[theo]{Corollary}

\theoremstyle{definition}

\numberwithin{equation}{section}
\theoremstyle{remark}
\newtheorem{rem}[theo]{Remark}
\newtheorem{exa}[theo]{Example}
\def\det{\mathrm{det}}

\def\Gal{\mathrm{Gal}}
\def\<{\langle}
\def\>{\rangle}

\def\GL{\mathrm{GL}}

\def\C{\mathbb{C}}
\def\N{\mathbb{N}}

\def\n{\eta}

\def\n'{\nu}

\def\tG{\widetilde{G}}

\begin{document}

\sloppy

\title[Hypertranscendence, the exponential case]{Hypertranscendence and  linear difference equations, the exponential case}
\author{Thomas Dreyfus}
\address{
Institut de Math\'ematique de Bourgogne, UMR 5584 CNRS, Universit\'e de Bourgogne, F-21000, Dijon, France}
\email{thomas.dreyfus@math.cnrs.fr}

\thanks{ This project has received funding from  the ANR De rerum natura project (ANR-19-CE40-0018). The IMB receives support from the EIPHI Graduate School (contract ANR-17-EURE-0002). }

\date{\today}

\begin{abstract} 
In this paper we study meromorphic  solutions of  linear shift difference equations with coefficients in $\C(x)$  involving the operator $\rho: y(x)\mapsto y(x+h)$, for some $h\in \C^*$. 
We prove that if $f$ is a solution of an algebraic differential equation, then $f$ belongs to a ring that is generated by periodic functions and exponentials.  Our proof is based on the parametrized difference Galois theory initiated by Hardouin and Singer. 
\end{abstract} 

\maketitle
\setcounter{tocdepth}{1}
\tableofcontents


\section*{Introduction}
Given a solution of a functional equation we might wonder if it is a solution of another kind of functional equation.  If the two functional equations are too different we would expect that very few functions are solutions of both functional equations.   The goal of this paper is  to state a result in that direction. 
Before recalling the state or the art, let us state our main result and its framework.  \par 
 Let $\rho$ be the automorphism of $\C(x)$ defined by $\rho: y(x)\mapsto y(x+h)$ with $h\in \C^*$.   Let $\mathcal{M}(\C)$ be the field of meromorphic functions on $\C$. 
We say that $f\in \mathcal{M}(\C)$ is differentially algebraic over  $\C(x)$, if there exist  $m\in \N$,  $0\neq P\in \C(x)[X_0,\dots,  X_m]$ such that $P(f, \partial_x f, \dots,  \partial_x^m f)=0$. We say that $f$ is differentially transcendental otherwise. 
 Let $C_h:=\{f\in \mathcal{M}(\C)| \rho(f)=f\}$ be the field of $h$-periodic functions.  The goal of this paper is to prove:

\begin{theo}\label{thm}
Let $n\in\N^*$, and let $A\in \mathrm{GL}_n (\C(x))$.  Let $\mathcal{Y}:= (f_1,\dots, f_n)^{\top}\in (\mathcal{M}(\C))^n$ be a solution of $\rho(Y)=AY$.  If every  $f_i$ is differentially algebraic over $\C(x)$,  then there exist $\ell\in \N^*$,  $\lambda_{1}\dots,\lambda_k \in \C$, such that  for all $1\leq i \leq n$,  $f_i \in C_{\ell h}(x)[e^{\lambda_1 x},\dots, e^{\lambda_k x}]$.  
\end{theo}
As a straightforward corollary, we find: 
\begin{coro}\label{cor}
Let $f\in \mathcal{M}(\C)$ such that 
$$a_0 f+\dots+ a_n \rho^n f=0,\quad a_i\in \C(x),  \quad a_0, a_n\neq 0.$$  If $f$ is differentially algebraic over $\C(x)$,  then  there exist $\ell\in \N^*$,  $\lambda_{1}\dots,\lambda_k \in \C$, such that  $f \in C_{\ell h}(x)[e^{\lambda_1 x},\dots, e^{\lambda_k x}]$.  
\end{coro}

\begin{proof}[Proof of Corollary \ref{cor}]
Since $\rho$ and $\partial_x$ commute, for all $i\in \N$,  $\rho^i (f)$ is differentially algebraic over $\C(x)$.   Then,  each entry of $\mathcal{Y}:= (f,\rho (f) \dots, \rho^{n-1}(f))^{\top}\in \mathcal{M}(\C)^n$ is differentially algebraic over $\C(x)$.  Since 

$$
\rho(\mathcal{Y} )=\begin{pmatrix}
0&1&0&\cdots&0\\
0&0&1&\ddots&\vdots\\
\vdots&\vdots&\ddots&\ddots&0\\
0&0&\cdots&0&1\\
-a_{0}/a_n& -a_{1}/a_n&\cdots & \cdots & -a_{n-1}/a_n
\end{pmatrix} \mathcal{Y},
$$ we conclude with Theorem~\ref{thm}. 
\end{proof}

The first proof of differential transcendence  of a function is due to Hölder,  see \cite{Ho}.  The author proved that the Gamma function is differentially transcendental.  His approach mainly uses the functional equation  $\rho (\Gamma)= x\Gamma$ (here $h=1$). \par 
 More recently in \cite{bezivin1993solutions}, it is proved that if $f\in \mathcal{M}(\C)$ is solution of a linear $\rho$-equation and a linear differential equation, then  there exist $\ell\in \N^*$,  $\lambda_{1}\dots,\lambda_k \in \C$, such that  $f \in C_{\ell h  }(x)[e^{\lambda_1 x},\dots, e^{\lambda_k x}]$.  The strategy of the authors is to study the singularities of the solutions of a linear $\rho$-equation,  linear differential equation, and prove that there is an incompatibility, unless the function is sufficiently simple.  Similar results were known for $q$-difference equations ($\rho : y(x)\mapsto y(qx)$, $q\in \C^*$,  $|q|\neq 1$) and Mahler equations (${\rho : y(x)\mapsto y(x^p)}$, $p\in \N_{\geq 2}$),   and were proved in an unified way in \cite{SS16}.  \par 
 The question of showing that a solution of a linear $\rho$-equation is not  solution of an algebraic differential equation is much more complicated.  Contrary to the case of a linear differential equation,  the number of unknowns grows too fast to hope to predict the potential equation in a brute-force search.  Another approach that fails is the study of the singularities, since a solution of an algebraic differential equation may have a very large set of singularities.   \par 
 Old results were known for affine equations of order one of the form ${\rho(y)=ay+b}$, when $\rho$ is  the shift operator considered above,  the  $q$-difference operator or a Mahler operator,  see for instance  \cite{Ho,Mo,Ma30,Ni84,rande1992equations,ishizaki1998hypertranscendency,Ha08,HS08,NG12}.
 \par 
 For higher order equations we need the Galois theory of difference equations to be able to prove similar results.  More precisely, given a linear $\rho$-system, we may associate a Galois group of matrices that encodes the algebraic (resp. algebraic and  differential) relations among the solutions.  The bigger this group is, and fewer relations there are.  So if we are able to prove that the Galois group is sufficiently large, we might expect that the solutions are differentially transcendental. The first attempt of this strategy could be found in \cite{Ha08} for a modern proof of Hölder's theorem of the differential transcendence of the Gamma function.  This approach has been generalized into a complete theory in \cite{HS08} and has been applied to prove that when the Galois group is big, the solutions are differentially transcendental,  see \cite{HS08,DHR18,DHR2,AS17,ArDR18,bostan2020differential,di2023inhomogeneous}.  The general idea behind the latter papers is that when the solutions are differentially algebraic and the Galois group is big,  then the functions are solutions of linear differential equations. We then use the 
results above mentioned of \cite{SS16}    to prove that they are too trivial and force the Galois group to be small,  leading to a contradiction.  The Galois theory of \cite{HS08} has been applied to a very different context to prove in \cite{DHRS} that some generating series of walks in the quarter plane are differentially transcendental.  \par
At this stage, we are able to prove that the solutions of linear $\rho$-equations are differentially transcendental or trivial only when the Galois group is big, or small (affine equations). It remains to treat the medium cases. This is the goal of \cite{adamczewski2021hypertranscendence} where it is proved that when the Galois group is medium, then the $\rho$-equation is equivalent to a smaller one. Then an induction proof on the rank of the equation allows the authors to prove that either the solution is differentially transcendental,  or it belongs to a small field.  The results of \cite{adamczewski2021hypertranscendence} are stated for meromorphic solutions at infinity and the shift operator,  for the $q$-difference operator, and for the Mahler operator.    For the shift case and meromorphic function on $\C$ which is the framework of this paper,  the statement is much more complicated to set for two reasons: 
\begin{itemize}
\item the exponential functions $e^{\lambda x}$, $\lambda\in \C$,  are meromorphic solutions of linear $\rho$-equations and linear differential equations. 
\item the meromorphic functions that are $\ell h$-periodic for some $\ell \in \N^*$ are solutions of linear $\rho$-equations, and some of them are solutions of  differential equations.  
\end{itemize}
Then, the set of functions that are both solutions of a linear $\rho$-equation and  differential equation is bigger than the ground field $\C(x)$. To avoid this  problem,  the authors of  \cite{adamczewski2021hypertranscendence} consider the situation where  $f$ is solution of a linear $\rho$-equation in coefficients in $\C(x)$ and belong to a field of meromorphic function $F\subset \mathcal{M}(\C)$,  with $\rho(F)=F$, $\{f\in F| \rho(f)=f\}=\C$,  and  for all 
$\lambda_{1}\dots,\lambda_k \in \C$,  $F\cap  \C(x,e^{\lambda_1 x},\dots, e^{\lambda_k x})=\C(x)$.  Then it is proved that either $f\in \C(x)$,  or $f$ is differentially transcendental.   Thus,  Theorem \ref{thm} generalizes this result in this context since we avoid the assumptions on $F$.  The strategy of the proof of Theorem~\ref{thm} is  in some sense similar to the one in the recent paper \cite{de2022algebraic} that deals with another framework. \par 
The paper is organized as follows. In Section \ref{secgal} we give a reminder of the difference Galois theory and in Section \ref{secgaldif}, the parametrized difference Galois theory.  In Section \ref{secirred}, we deal with difference Galois groups that are irreducible and connected, and in Section  \ref{secfin} we complete the proof of Theorem~\ref{thm}.

\paragraph{\textbf{Aknowledgment}} The author would like to thank the anonymous referee for the helpful comments.
\section{Difference setting}\label{secgal}

The goal of this section is to give a short review of the main results of difference algebra and difference Galois theory that will be used in this paper.  We refer to \cite{Cohn} for more details on difference algebra and \cite{vdPS97} for more details on difference Galois theory. 
In what follows,  all fields are of characteristic zero,  and all rings are unitary.  \par 
A difference ring is a ring $(R,\rho)$ equipped with a noncyclic automorphism. 
We define similarly  the notion of difference fields,  difference algebras,  etc...  
When no confusion arises,  we will denote by $R$ the difference ring $(R,\rho)$.  The ring of constants is defined by $R^{\rho}:=\{r\in R | \rho(r)=r\}$.  If $R$ is a field,  $R^{\rho}$ is also a field and will be called the field of constants.   
\begin{exa}
If we consider the notations of the introduction,   $(\C(x),\rho)$ and $(\mathcal{M}(\C),\rho)$ are difference fields, and we have $\C(x)^{\rho}=\C$ and $\mathcal{M}(\C)^{\rho}=C_h$. 
\end{exa}
A $\rho$-ideal $I\subset R$ is an ideal such that $\rho (I)\subset I$. We say that the difference ring is $\rho$-simple if the only $\rho$-ideals are $\{0\}$ and $I$.  \par 
Let $\mathbf{k}$ be a difference field.  Let us assume that  $\mathbf{C}:=\mathbf{k}^{\rho}$ is     algebraically closed.  We consider the difference system 
\begin{equation}\label{eq1}
\rho(Y)=AY,  \quad A\in \mathrm{GL}_n(\mathbf{k}). 
\end{equation}
A Picard-Vessiot ring extension for \eqref{eq1} over $\mathbf{k}$ is a difference ring extension $\mathcal{R}|\mathbf{k}$ such that 
\begin{itemize}
\item There exists $U\in \mathrm{GL}_{n}(\mathcal{R})$,  such that $\rho(U)=AU$, such matrix is called a fundamental	 matrix;
\item $\mathcal{R}=\mathbf{k}[U,1/\det(U)]$;
\item $\mathcal{R}$ is a simple difference ring.  
\end{itemize}

A Picard-Vessiot ring extension exists and is unique up to isomorphism of $\mathbf{k}$-$\rho$-algebras.  

Given a Picard-Vessiot ring extension $\mathcal{R}|\mathbf{k}$,  the Picard-Vessiot extension $\mathcal{Q}$ is the total ring of fractions of $\mathcal{R}$.  We have $\mathcal{Q}^{\rho}=\mathcal{R}^{\rho}=\mathbf{k}^{\rho}=\mathbf{C}$.   
 We define the difference Galois group as  the group of ring automorphisms of $\mathcal{Q}$, leaving $\mathbf{k}$ invariant and commuting with $\rho$, that is 
$$\mathrm{Gal}(\mathcal{Q}|\mathbf{k})=\{ \sigma \in \mathrm{Aut}(\mathcal{Q}|\mathbf{k})| \sigma\circ \rho=\rho\circ \sigma\}.$$

For any fundamental matrix $U \in \GL_n(\mathcal{Q})$, an easy computation shows that 
$U^{-1}\sigma(U) \in \GL_{n}(\mathbf{C})$ for all $\sigma \in \Gal(\mathcal{Q}|\mathbf{k})$. 
By  \cite[Theorem~1.13]{vdPS97}, the faithful representation
\begin{eqnarray*}
\Gal(\mathcal{Q}|\mathbf{k}) & \rightarrow & \GL_{n}(\mathbf{C}) \\ 
  \sigma & \mapsto & U^{-1}\sigma(U)
\end{eqnarray*}
identifies $\Gal(\mathcal{Q}|\mathbf{k}) $ with a linear algebraic subgroup $G\subset\GL_{n}(\mathbf{C})$. 
Choosing another fundamental matrix of solutions $U$ leads to a conjugate representation. 

The following result is a slight adaptation of \cite[Lemma 2.3]{DHR18} without the assumption that $C$ is algebraically closed. Although the proof is mostly the same, we are going to give it in order to be self contained.

\begin{lem}\label{lem7}
Let $F$ be a difference field with field of constants that is $\mathbf{C}$. Let $\widetilde{\mathbf{C}}$ be a field extension of $\mathbf{C}$. Consider the ring  $\widetilde{\mathbf{C}}\otimes_{\mathbf{C}} F$, equipped with a structure of difference ring via $\rho(c\otimes f)=c\otimes \rho(f)$, for all $c\in \widetilde{\mathbf{C}}$, $f\in F$. Then, $\widetilde{\mathbf{C}}\otimes_{\mathbf{C}} F$ is a simple difference ring. Furthermore, $\mathcal{Q}$, the  total quotient  ring of $\widetilde{\mathbf{C}}\otimes_{\mathbf{C}} F$, satisfies $\mathcal{Q}^{\rho}=\widetilde{\mathbf{C}}$.  
\end{lem}
 Note that in general, $\widetilde{\mathbf{C}}\otimes_{\mathbf{C}} F$ is not an integral domain.
\begin{proof}
For $c_1,\dots,c_{\kappa}\in \widetilde{\mathbf{C}}$ that are $\mathbf{C}$-linearly independent, for $f_i\in F$, we claim that $0=\sum_{i=1}^{\kappa} c_i \otimes f_i$ implies $f_i=0$ for all $i$. To the contrary, let $c_i\in \widetilde{\mathbf{C}}$ that are $\mathbf{C}$-linearly independent, $f_i\in F$, not all zero, such that $0=\sum_{i=1}^{\kappa} c_i \otimes f_i$. Furthermore, without loss of generality we may assume that  $\kappa$ is minimal, that is if we have  $0=\sum_{i=1}^{\kappa'} c'_i \otimes f'_i$, with $c'_i\in \widetilde{\mathbf{C}}$ that are $\mathbf{C}$-linearly independent, $f'_i\in F$, not all zero, then $\kappa'\geq\kappa$.  By minimality of $\kappa$, $f_{\kappa}\neq 0$. Dividing by $f_{\kappa}$ we may reduce to the case where $f_{\kappa}=1$. Applying $\rho$ we find 
$0=\sum_{i=1}^{\kappa} c_i \otimes \rho(f_i)$, and substracting with the first relation, we obtain $0=\sum_{i=1}^{\kappa} c_i \otimes (\rho(f_i)-f_i)$. Since $f_{\kappa}=1$, $\rho(f_{\kappa})-f_{\kappa}=0$, and by minimality of $\kappa$, we have $\rho(f_i)=f_i$, for all $i$. Then for all $i$, $f_i\in F^{\rho}=\mathbf{C}$ and $0=\sum_{i=1}^{\kappa} c_i  f_i$. This contradicts the fact that the $c_i$ are $\mathbf{C}$-linearly independent. This proves our claim. \par 
Consider $I$ a difference ideal of $\widetilde{\mathbf{C}}\otimes_{\mathbf{C}} F$ different from $(0)$. Let us prove that $I=\widetilde{\mathbf{C}}\otimes_{\mathbf{C}} F$. Let $w=\sum_{i=1}^{\kappa}c_i \otimes f_i\in I$ be a nonzero element of $I$ with $c_i\in \widetilde{\mathbf{C}},f_i \in F$ and $\kappa$ minimal. By minimality of $\kappa$, $f_{\kappa}\neq 0$ and the $c_i$ are $\mathbf{C}$-linearly independent. By replacing  $w$ with $w/f_{\kappa}\in I$ we may reduce to the case where $f_{\kappa}=1$.   Since $I$ is a difference ring, $\rho(w)-w\in I$. But $\rho(w)-w=\sum_{i=1}^{\kappa}c_i \otimes (\rho(f_i)-f_i)=\sum_{i=1}^{\kappa-1}c_i \otimes (\rho(f_i)-f_i)$. If $\rho(w)-w=0$ by the above claim for all $i$, $\rho(f_i)=f_i$. If $\rho(w)-w\neq 0$, 
by minimality of $\kappa$, for all $i$, $\rho(f_i)=f_i$. Then for all $i$, $f_i\in F^{\rho}=\mathbf{C}$ and then $w=\sum_{i=1}^{\kappa}c_i f_i \otimes 1$ is invertible in $\widetilde{\mathbf{C}}\otimes_{\mathbf{C}} F$. Hence $1\in I$ and $I=\widetilde{\mathbf{C}}\otimes_{\mathbf{C}} F$, proving that $\widetilde{\mathbf{C}}\otimes_{\mathbf{C}} F$ is a simple difference ring. \par
Let $c\in \mathcal{Q}^{\rho}$ and let us prove that $c\in \widetilde{\mathbf{C}}$. Consider $I=\{d\in \widetilde{\mathbf{C}}\otimes_{\mathbf{C}} F| cd\in \widetilde{\mathbf{C}}\otimes_{\mathbf{C}}F\}$. Since  $ \widetilde{\mathbf{C}}\otimes_{\mathbf{C}}F$ is a ring we obtain that $I$ is an ideal. Let us prove that it is a difference ideal. Let $d\in I$. Then $cd\in  \widetilde{\mathbf{C}}\otimes_{\mathbf{C}}F$. Since $\rho(c)=c$ and $ \widetilde{\mathbf{C}}\otimes_{\mathbf{C}}F$ is a difference ring, $\rho(cd)=c\rho(d)\in  \widetilde{\mathbf{C}}\otimes_{\mathbf{C}}F$, proving that $\rho(d)\in  \widetilde{\mathbf{C}}\otimes_{\mathbf{C}}F$. Then $I$ is a difference ideal different from $(0)$. Since $\widetilde{\mathbf{C}}\otimes_{\mathbf{C}}F$ is a simple difference ring, we have $I=\widetilde{\mathbf{C}}\otimes_{\mathbf{C}}F$. Then $1\in I$, proving that $c\in \widetilde{\mathbf{C}}\otimes_{\mathbf{C}}F$. 
Let $c=\sum_{i=1}^{\kappa} c_i \otimes f_i\in (\widetilde{\mathbf{C}}\otimes_{\mathbf{C}} F)^{\rho}$ with $c_i\in \widetilde{\mathbf{C}}$, that are $\mathbf{C}$-linearly independent and $f_i\in F$.  
 Let us prove that $c\in \widetilde{\mathbf{C}}$. 
Then, $0=\rho(c)-c=\sum_{i=1}^{\kappa} c_i \otimes (\rho(f_i)-f_i)$. 
With the above claim, $\rho(f_i)-f_i=0$, for all $i$. Hence $f_i\in F^{\rho}=\mathbf{C}$ and then $c\in \widetilde{\mathbf{C}}$.
\end{proof}

\subsection*{Specific results for shift equations}
In this paragraph,  we consider the  field $\C(x)$ equipped with the automorphism $\rho: y(x) \mapsto y(x+h)$,  $h\in \C^*$, and give some specific results. 
\begin{rem}\label{rem1}
Note that $\mathcal{Y}(x)$ is a meromorphic solution of   $\mathcal{Y}(x+h)=A(x)\mathcal{Y}(x)$ if and only if $\mathcal{Z}(x):=\mathcal{Y}(hx)$ is a meromorphic solution of $\mathcal{Z}(x+1)=A(hx)\mathcal{Z}(x)$.  
Furthermore,  an entry of $\mathcal{Y}$  is differentially algebraic over $\C(x)$ if and only if the corresponding entry of  $\mathcal{Z}$ is differentially algebraic  over $\C(x)$. 
 Hence, we may apply the specific results of \cite{Pra,HS08}, originally  stated for $h=1$, for a general $h\in \C^*$ in this paper.  \end{rem}
 Fix $B\in \mathrm{GL}_{n}(\C(x)) $; then there exists  $V \in \mathrm{GL}_{n}(\mathcal{M}(\C))$ such that ${\rho( V)=BV}$ (\cite[Theorem 1]{Pra}).  Let $\overline{C_h}$ be the algebraic closure of $C_h$ and  consider 
$\overline{C_h}\otimes \mathcal{M}(\C)$ equipped with the structure of a $\rho$-ring via $\rho(c\otimes f)=c\otimes \rho(f)$.

The following result  will be used in the sequel.
\begin{lem}\label{lem5}
 The ring $\mathcal{R}=\overline{C_{h}}(x)[V,1/\det(V)]\subset \overline{C_h}\otimes \mathcal{M}(\C)$ is a Picard-Vessiot ring extension for $\rho(Y)=BY$ over $\overline{C_{ h}}(x)$.
\end{lem}

\begin{proof}
We  apply Lemma \ref{lem7} with 
$F=C_{h}(x,V)\subset \mathcal{M}(\C)$, $\widetilde{\mathbf{C}}=\overline{C_h}$, and $\mathbf{C}=C_h$. We obtain a simple difference ring 
$\overline{C_h}\otimes_{C_h} C_{h}(x,V)$. The nilpotent elements form a difference ideal, and by simplicity, it is $(0)$ proving that $\overline{C_h}\otimes_{C_h} C_{h}(x,V)$ has no nilpotent elements. Then, the total quotient  ring  has no nilpotent elements too. By Lemma \ref{lem7}, the latter has field of constants that is  also $\overline{C_h}$. As a total quotient  ring, every nonzero divisor is invertible.
 We conclude with \cite[Propositions~6.17]{HS08}, applied with the trivial derivation, see also \cite[Corollary~1.24]{vdPS97}, that $\mathcal{R}=\overline{C_{h}}(x)[V,1/\det(V)]$ is a Picard-Vessiot ring extension for $\rho(Y)=BY$ over $\overline{C_{ h}}(x)$. 
\end{proof}

The change of variables $Z:=TY$,  $T\in \GL_{n}(\C(x))$, transforms the system  $\rho (Y)=AY$ with  $A\in  \GL_{n}(\C(x))$ into $\rho (Z)=BZ$ where $B:=\rho(T)AT^{-1}\in  \GL_{n}(\C(x))$.  Let  $A,B\in  \GL_{n}(\C(x))$.  We will say that $[A]$ and $[B]$ are equivalent over $\C(x)$ if there exists   $T\in \GL_{n}(\C(x))$ such that $B=\rho(T)AT^{-1}$.  The following theorem has been proved in \cite[Propositions~1.20 and~1.21]{vdPS97}.  

\begin{theo}
Let $G:=\mathrm{Gal}(\mathcal{Q}|\C(x))$ that we identify as an algebraic subgroup of $\mathrm{GL}_n (\C)$.  There exists $T\in \mathrm{GL}_n(\C(x))$ such that $\rho(T)AT^{-1}\in G(\C(x))$. 
\end{theo}

\section{Differential setting}\label{secgaldif}

We refer to \cite{HS08} for more details on what follows.  A differential ring $(R,\delta)$ is a ring equipped with a derivation, that is an additive morphism satisfying the Leibnitz rule $\delta (fg)=f\delta(g)+\delta (f)g$.   We define similarly  the notion of differential fields,  differential algebras,  etc...  The ring of $\delta$-constants  of $R$ is defined by 
$$
R^\delta =\{ r \in R \mid  \delta(r)= 0\}\,.
$$ 
If $R$ is a field,  $R^{\delta}$ is also a field and will be called the field of constants.   Let $(R,\delta)$ be a differential ring extension for the differential field $(\mathbf{k},\delta)$. We say that $f\in R$ is differentially algebraic over $\mathbf{k}$ if there exist $m\in \N$,  and $0\neq P\in \mathbf{k}[X_0,\dots, X_m]$ such that $P(f, \delta f ,\dots, \delta^m f)=0$. We say that $f$ is differentially transcendental over $\mathbf{k}$ otherwise. \par 
A $(\rho,\delta)$-ring $(R,\rho, \delta)$ is a ring equipped with an automorphism $\rho$  and a derivation $\delta$ that commutes with $\rho$.  We define similarly the notion of $(\rho,\delta)$-fields,   etc...
\begin{exa}
If we consider the notation of the introduction,   $(\C(x),\rho, \partial_x)$ and $(\mathcal{M}(\C),\rho, \partial_x)$ are $(\rho,\partial_x)$-fields, and we have $\C(x)^{\partial_x}=\mathcal{M}(\C)^{\partial_x}=\C$. 
\end{exa}
We say that $I$ is a $(\rho,\delta)$-ideal of $R$ if $I$ is an ideal  such that $\rho(I)\subset I$ and $\delta(I)\subset I$.   We say that $(R,\rho,\delta)$ is $(\rho,\delta)$-simple if its only $(\rho,\delta)$-ideals are $\{0\}$ and $R$.

Let us denote by $R\{X_1,\dots, X_n\}_{\delta}$ the ring  of differential polynomials in the indeterminates $\delta^{i}X_j$.
  We recall that a  differential field $(L,\delta)$ is called  differentially closed
or $\delta$-closed if, for every set of $\delta$-polynomials $\mathcal F$,  
the system of $\delta$-equations $\mathcal F=0$ has a solution in some $\delta$-field extension of $L$
 if and only if it has a solution in $L$.  There always exists a differential field extension that is differentially closed.  \par 
 Assume that $\mathbf{k}$ is a $(\rho,\delta)$-field with $\mathbf{C}=\mathbf{k}^\rho$ differentially closed and let $A\in \mathrm{GL}_n(\mathbf{k})$.  
A $(\rho,\delta)$-Picard-Vessiot ring extension for $\rho(Y)=AY$ over $\mathbf{k}$ is a $(\rho,\delta)$-ring extension $\mathcal{S}|\mathbf{k}$ such that 
\begin{itemize}
\item There exists $U\in \mathrm{GL}_{n}(\mathcal{S})$;
\item $\mathcal{S}=\mathbf{k}\{U,1/\det(U)\}_{\delta}$;
\item  $\mathcal{S}$ is a simple $(\rho,\delta)$-ring.  
\end{itemize}

A $(\rho,\delta)$-Picard-Vessiot ring extension exists and is unique up to isomorphisms of $\mathbf{k}$-$(\rho,\delta)$-algebras.  
Given a $(\rho,\delta)$-Picard-Vessiot ring extension $\mathcal{S}|\mathbf{k}$,  the $(\rho,\delta)$-Picard-Vessiot extension $\mathcal{Q}_S$ is the total ring of fractions of $\mathcal{S}$.  We have $\mathcal{Q}_S^{\rho}=\mathcal{S}^{\rho}=\mathbf{k}^{\rho}=\mathbf{C}$.
We define the $(\rho,\delta)$-Galois group as  the group of ring automorphisms of $\mathcal{Q}_S$, leaving $\mathbf{k}$ invariant and commuting with $\rho$ and $\delta$, that is 
$$\mathrm{Gal}^{\delta}(\mathcal{Q}_S|\mathbf{k})=\{ \sigma \in \mathrm{Aut}(\mathcal{Q}_S|\mathbf{k})| \sigma \circ \rho=\rho\circ \sigma,\quad  \sigma\circ \delta=\delta\circ \sigma\}.$$

For any fundamental matrix $U \in \GL_n(\mathcal{Q}_S)$, an easy computation shows that 
$U^{-1}\sigma(U) \in \GL_{n}(\mathbf{C})$ for all $\sigma \in \Gal^{\delta}(\mathcal{Q}_S|\mathbf{k})$. 
By  \cite[Theorem~2.6]{HS08}, the faithful representation 
\begin{eqnarray*}
\Gal^{\delta}(\mathcal{Q}_S |\mathbf{k}) & \rightarrow & \GL_{n}(\mathbf{C}) \\ 
  \sigma & \mapsto & U^{-1}\sigma(U)
\end{eqnarray*}
identifies $\Gal^{\delta}(\mathcal{Q}_S|\mathbf{k}) $ with a linear differential algebraic subgroup ${H\subset\GL_{n}(\mathbf{C})}$,  that is a group of matrices whose entries satisfy a set of  algebraic differential relations. 
Choosing another fundamental matrix of solutions $U$ leads to a conjugate representation.

\subsection*{Specific results for shift equations}
The field $C_h$ may be equipped with a structure of differential field with the derivation $\delta:=\partial_x$.\par 
 Let $\widetilde{C_h}$ be a differentially closed field containing $C_{h}$.   Consider 
 $\widetilde{C_h}(x)$,  that is equipped with a structure of $(\rho,\delta)$-field with $\left(\widetilde{C_h}(x)\right)^{\rho}=\widetilde{C_h}$,  $\rho(x)=x+h$ and $\delta (x)=1$.   Note that $\rho$ and $\delta$ commute.  
Consider 
$\widetilde{C_h}\otimes \mathcal{M}(\C)$ equipped with the structure of a $(\rho,\delta)$-ring via $\rho(c\otimes f)=c\otimes \rho(f)$ and $\delta(c\otimes f)=\delta(c)\otimes f+ c\otimes \delta (f)$. Note that  $(\widetilde{C_h}\otimes \mathcal{M}(\C))^{\rho}=\widetilde{C_h}$.  
 
    The following result that is the analogue of Lemmas \ref{lem7} and \ref{lem5}, and its proof is totally similar.

\begin{lem}\label{lem6}
Let $B\in \mathrm{GL}_{n}(\C(x))$ and  let $V\in \mathrm{GL}_{n}(\mathcal{M}(\C))$ such that $\rho(V)=BV$.  The ring $\mathcal{S}=\widetilde{C_h}(x) \{ V,1/\det(V) \}_{\delta}$ is a $(\rho,\delta)$-Picard-Vessiot ring extension for $\rho(Y)=BY$ over $\widetilde{C_{ h}}(x)$. 
\end{lem}

  Let $B\in \mathrm{GL}_{n}(\C(x))$.  Given $\ell \in \N^*$,  we may iterate the difference system $\rho(Y)=BY$ by considering $\rho^{\ell}(Y)=B_{[\ell]}Y$,  where $B_{[\ell]}= \rho^{\ell -1}(B)\times \dots \times B$. 
Let $G_{[\ell]}$ be the difference Galois group of the system $\rho^{\ell}(Y)=B_{[\ell]}Y$ and $G$ be the difference Galois group of $\rho(Y)=BY$.\par 
The following lemma is a slight adaptation of  \cite[Proposition~4.6]{DHR2} in the particular case where the parametric operator is the identity and $\kappa=1$,  see also  \cite[Proposition~4.10]{adamczewski2021hypertranscendence}.

\begin{lem}\label{lem2}
Let $B\in \mathrm{GL}_{n}(\C(x))$. Let $1\leq \kappa \leq n$, and  let $\mathcal{Y}_1,\dots,\mathcal{Y}_{\kappa} \in \mathcal{M}(\C)^n$ vectors solution of $\rho(Y)=BY$ that are $\mathcal{M}(\C)$-linearly independent.
 Then, there exists $\ell\geq 1$, such that 
\begin{itemize}
\item there exists a $(\rho^{\ell},\delta)$-Picard-Vessiot extension for $\rho^{\ell}(Y)=B_{[\ell]}Y$ over $\widetilde{C_h}(x)$, that is a field,  with fundamental matrix that has  coefficients in  $\mathcal{M}(\C)$, and admits $\mathcal{Y}_1,\dots,\mathcal{Y}_{\kappa}$ as first $\kappa$ vector solutions. 
\item $G_{[\ell]}$ is the connected component of the identity of $G$. 
\end{itemize}
\end{lem}

\begin{proof}
As we can see in the proof of \cite[Proposition~4.6]{DHR2}, there exists  $\ell\geq 1$, such that the $(\rho^{\ell},\delta)$-Picard-Vessiot extension for $\rho^{\ell}(Y)=B_{[\ell]}Y$ over $\widetilde{C_h}(x)$, is a field and  $G_{[\ell]}$ is the connected component of the identity of $G$. By \cite[Theorem 1]{Pra}, let $V\in \mathrm{GL}_n (\mathcal{M}(\C))$ solution of  $\rho^{\ell}(Y)=B_{[\ell]}Y$. Since $V$ is invertible, the $n$ columns of $V$ form a $\mathcal{M}(\C)
$-vector space of dimension $n$.  Due to the fact that $\mathcal{Y}_1,\dots,\mathcal{Y}_{\kappa}$ are $\mathcal{M}(\C)$-linearly independent, we may use the incomplete basis theorem to find $\mathcal{M}(\C)$-linearly independent vector  $\mathcal{Y}_1,\dots,\mathcal{Y}_{\kappa},Y_{\kappa+1}, \dots,Y_{n} \in \mathcal{M}(\C)^n$, which are solutions of $\rho^{\ell}(Y)=B_{[\ell]}Y$, and consider $V'\in \mathrm{GL}_n (\mathcal{M}(\C))$ be the corresponding matrix. It is a fundamental matrix of  $\rho^{\ell}(Y)=B_{[\ell]}Y$ with $\mathcal{Y}_1,\dots,\mathcal{Y}_{\kappa}$ as first columns.  Then, without loss of generality,  we may replace $V$ by this fundamental matrix and reduce to the case where $\mathcal{Y}_1,\dots,\mathcal{Y}_{\kappa}$ are the first columns of $V$.
By Lemma~\ref{lem6} the ring $\widetilde{C_h}(x)\{V,1/\det(V)\}_{\delta}$ is a $(\rho^{\ell},\delta)$-Picard-Vessiot ring extension for $\rho^{\ell}(Y)=B_{[\ell]}Y$ over $\widetilde{C_h}(x)$. Since the $(\rho^{\ell},\delta)$-Picard-Vessiot extension is a field,  $\widetilde{C_h}(x)\{V,1/\det(V)\}_{\delta}$ is an integral domain and its field of fraction is isomorphic to the $(\rho^{\ell},\delta)$-Picard-Vessiot extension.  
\end{proof}


\section{Irreducible Galois group}\label{secirred}
Recall that we consider the system   $\rho(Y)=AY$ with $A\in \mathrm{GL}_n (\C(x))$.  
By  \cite[Theorem 1]{Pra},  there exists $U\in \mathrm{GL}_{n}(\mathcal{M}(\C))$ such that $\rho(U)=AU$.  Let $\mathcal{Q}_{\C}$ be a Picard-Vessiot extension for $\rho(Y)=AY$ over $\C(x)$. 
In this section, following \cite[Section~5.2]{adamczewski2021hypertranscendence} we consider the case where $\Gal(\mathcal{Q}_{\C} |\C(x))$,  the difference Galois group of $\rho(Y)=AY$ over $\C(x)$,  seen as an algebraic subgroup of $\mathrm{GL}_n(\C)$,  is irreducible and connected.   Recall that an algebraic subgroup $G\subset  \mathrm{GL}_n(\C)$  is said to be irreducible;  if and only if for all  $\C$-vector spaces $V\subset \C^n$,  $G(V)\subset V$ implies that $V$ is either $\{0\}$ or $\C^n$.   The proof is very similar to \cite[Section~5.2]{adamczewski2021hypertranscendence},  and only the points where this proof is really different to the one of  \cite[Section~5.2]{adamczewski2021hypertranscendence} will be detailed.
\begin{propo}\label{prop: colonne}
Let us assume that $n\geq 2$.  If  $\Gal(\mathcal{Q}_{\C} |\C(x))$ is irreducible and connected,  then every column of $U$
contains at least one element that is 
differentially transcendental over $\widetilde{C_h}(x)$.  
\end{propo}

As a key argument, we will use the following result due to \cite[Lemma~5.1]{AS17}.  It says that the $(\rho,\delta)$-Galois group must be as big as possible when 
the difference Galois group has an identity component that is semisimple.

\begin{propo}\label{propo:AS}
Let $B\in \mathrm{GL}_n (\C(x))$.  
We let $\widetilde{G}\subset \mathrm{GL}_{n}(\widetilde{C_h})$ 
denote the difference Galois group of $\rho(Y)=BY$ over $\widetilde{C_h}(x)$, and $H\subset \mathrm{GL}_{n}(\widetilde{C_h})$ denote the $(\rho,\delta)$-Galois group of $\rho(Y)=BY$ over $\widetilde{C_h}(x)$. 
 If the identity component of $\widetilde{G}$ 
 is semisimple, then $H=\widetilde{G}$.
\end{propo}

\begin{proof}[Proof of Proposition \ref{prop: colonne}]  
Let us argue by contradiction assuming that there exists  one column of the fundamental matrix $U$ whose coordinates are all 
differentially algebraic over $\widetilde{C_h}(x)$.  By Lemma \ref{lem5},  $\overline{C_h}(x)[U,1/\det(U)] $ is a Picard-Vessiot ring extension for $\rho(Y)=AY$ over $\overline{C_{ h}}(x)$ and by Lemma~\ref{lem6}, $\widetilde{C_h}(x) \{ V,1/\det(V) \}_{\delta}$ is a $(\rho,\delta)$-Picard-Vessiot ring extension for $\rho(Y)=AY$ over $\widetilde{C_{ h}}(x)$.  The difference Galois group of $\rho(Y)=AY$ over $\widetilde{C_h}(x)$ is $\Gal(\mathcal{Q}_{\C} |\C(x))(\widetilde{C_h})$ and is therefore irreducible. 
Hence,   the same reasoning as the proof of \cite[Proposition~5.4]{adamczewski2021hypertranscendence} shows that 
 all entries of $U$ are 
differentially algebraic  over $\widetilde{C_h}(x)$.   Then, $\det(U)$ is also differentially algebraic  over $\widetilde{C_h}(x)$.
We observe that the determinant $\det(U)$ is solution to the equation 
\begin{equation}\label{eq: detA}
\rho(y)=\det(A)y
\end{equation}
and  that 
the difference Galois group of this equation over $\C(x)$ 
 is the group $\det(\Gal(\mathcal{Q}_{\C} |\C(x)))$.   We claim that  $\det(U)$ is differentially algebraic over $C_{h}(x)$.   We find that $\frac{\partial_x \det(U)}{\det(U)}$ is solution of $\rho(y)-y=\frac{\partial_x \det(A)}{\det(A)}$.  Let $\widehat{b}(x):=\frac{\partial_x \det(A)}{\det(A)}\in \C(x)$. By Remark \ref{rem1},  we may use \cite[Proposition~3.1]{HS08}, to deduce that  there exist a nonzero linear differential operator $\widetilde{L}$ with coefficients in $\widetilde{C_h}$,  and $\widetilde{g}\in \widetilde{C_h}(x)$, such that $\widetilde{L}(\widehat{b})=\widetilde{g}(x+h)-\widetilde{g}(x)$.   Since $\widetilde{L}$ has coefficients in $\widetilde{C_h}$,  it follows that the poles of $\widetilde{L}(\widehat{b})$, seen as an element of $\widetilde{C_h}(x)$ are the poles of $\widehat{b}$, which are in $\C$, because   $\widehat{b}(x)\in \C(x)$. With $0\neq \widetilde{L}(\widehat{b})=\widetilde{g}(x+h)-\widetilde{g}(x)$, it follows that the poles of $\widetilde{g}(x)\in \widetilde{C_h}(x)$ are also in $\C$.   Taking the 
 partial fraction decomposition  yields that the equation $\widetilde{L}(\widehat{b})=\widetilde{g}(x+h)-\widetilde{g}(x)$ is equivalent to a polynomial equation with coefficients in $\C$.  Since $\C$ is algebraically closed it has a solution in $\C$, proving the existence of a nonzero linear differential operator $L$ with coefficients in $\C$,  and $g'\in \C(x)$, such that $L(\widehat{b})=g'(x+h)-g'(x)$. Then,   $L(\frac{\partial_x \det(U)}{\det(U)}( x))-g'(x)$ is $\rho$-invariant.  This shows that $\det(U)(x)$ is differentially algebraic over $C_{h}(x)$. By Remark \ref{rem1},  we may use 
 \cite[Corollary~3.4]{HS08}, to deduce that there exist some nonzero elements 
$c\in \C$ and $g\in \C(x)$ such that $\det(A)=c\rho(g)/g$.  Let $d\in \C$ such that $e^{dh}=c^{-1/n}$ and consider the matrix 
$V=Ue^{dx}\in \mathrm{GL}_{n}(\mathcal{M}(\C))$. It is solution of $\rho(Y)=BY$ with $B=c^{-1/n}A$.  We have $\det(B)=\rho(g)/g$.  By Lemma \ref{lem6}, 
$\widetilde{C_h}(x)\{V,1/\det(V)\}_{\delta}$  is a $(\rho,\delta)$-Picard-Vessiot ring extension for $\rho(Y)=BY$ over $\widetilde{C_h}(x)$. 
 Let $G$ be the difference Galois group of $\rho (Y)=BY$ over $\C(x)$,  $\widetilde{G}$ the difference Galois group of $\rho (Y)=BY$ over $\widetilde{C_h}(x)$, and $H$ the $(\rho,\delta)$-Galois group  over $\widetilde{C_h}(x)$.  
Same reasons as in the proof of   \cite[Proposition 5.4]{adamczewski2021hypertranscendence} show that $G$ is irreducible,   the connected component $G^0$ of the identity of $G$ is also irreducible,  and there exists $\ell\geq 1$ such that  
the difference Galois group of $\rho^\ell (Y)=B_{[\ell]}Y$ over $\C(x)$ is  $G^0$. 
Furthermore,  $\det (B_{[\ell]})= \rho^{\ell}(g)/g$ and therefore,  $G^0 \subset \mathrm{SL}_{n}(\C)$.  
With  \cite[Lemma 4.2]{adamczewski2021hypertranscendence} we find that $G^0 $ is primitive (we refer to the latter paper for the notion of primitive groups that will not be used again in the sequel).  By \cite[Proposition~2.3]{SU93},  we finally obtain that  $G^0 $ is semisimple. 
Then,  $\tG^0 =G^0 (\widetilde{C_h})$, the identity component of $\tG$,  is semisimple too. 

We infer from Proposition \ref{propo:AS} 
that $H=\tG$.  We recall that all entries of  $U$ are differentially algebraic  over $\widetilde{C_h}(x)$.  Then the same holds for ${V=Ue^{dx}}$.  The same reasoning as in the proof of   \cite[Proposition~5.4]{adamczewski2021hypertranscendence} shows that $G$ is a finite group.  
Since it is connected,  we deduce that $G^0=\{ \mathrm{I}_{n}\}$, 
where we let $\mathrm{I}_{n}$ denote the identity 
matrix of size $n$. Since by assumption $n\geq 2$, this provides a contradiction with the fact
that $G^0$ is irreducible. 
\end{proof}

\section{Proof of the main result}\label{secfin}
We are now ready to prove Theorem \ref{thm}.  Recall that 
  $\mathcal{Y}:= (f_1,\dots, f_n)^{\top}\in (\mathcal{M}(\C))^n$ is a solution of $\rho(Y)=AY$, with $A\in \mathrm{GL}_n (\C(x))$.
Let ${C^{\infty}_{h}=\displaystyle \cup_{\ell=1}^{\infty} C_{\ell h} }$ and let $R_{h,\exp}$ be the ring of   $C^{\infty}_h (x)$-linear combinations between $\{e^{\lambda_i x}, \lambda_i \in \C\}$.   We have to prove that when all the entries of $\mathcal{Y}$ are differentially algebraic over $\C(x)$,  they all belong to $R_{h,\exp}$.
If $\mathcal{Y}$ is zero the result is clear.  So let us assume that $\mathcal{Y}$ is nonzero. 
 Let $\mathcal{Q}$ be the Picard-Vessiot extension for $\rho(Y)=AY$ over $\overline{C_{h}}(x)$ and let $\Gal(\mathcal{Q} |\overline{C_{h}}(x))$ be the difference Galois group.  
By Lemma~\ref{lem2},  with $\delta=0$,  
there exists $\ell\geq 1$, such that 
\begin{itemize}
\item we have a Picard-Vessiot extension for $\rho^{\ell}(Y)=A_{[\ell]}Y$ over $(\overline{C_{h}}(x),\rho^{\ell})$ that is a field;
\item  the fundamental matrix $U$ of the Picard-Vessiot extension over $(\overline{C_h} (x),\rho^{\ell})$ has coefficients in $\mathcal{M}(\C)$ and  admits $\mathcal{Y}$ as first column;
\item the difference Galois group of $\rho^{\ell}(Y)=A_{[\ell]}Y$  over $\overline{C_{h}}(x)$ is 
the component of the identity of $\Gal(\mathcal{Q} |\overline{C_{h}}(x))$.
\end{itemize} 
  Note that  $(\C(x),\rho^{\ell},\delta)$ is  a $(\rho^{\ell},\delta)$-field and $R_{\ell h, \exp} \subset R_{h, \exp}$ is a  $(\rho^{\ell},\delta)$-ring. 
So,  without loss of generality, we may replace $A$ by  $A_{[\ell]}$,  $h$ by $\ell h$, and reduce to the case where the Picard-Vessiot extension $\mathcal{Q}$ is a field,  the difference Galois group over $\overline{C_{h}}(x)$ is connected,  and the Picard-Vessiot extension for  $\rho(Y)=AY$ over $\overline{C_{h}}(x)$  admits a fundamental matrix with coefficients in $\mathcal{M}(\C)$ and with $\mathcal{Y}$ as first column.   Let $\mathcal{Q}_{\C}$ be the Picard-Vessiot extension for $\rho(Y)=AY$ over $\C(x)$ and let $\Gal(\mathcal{Q}_{\C} |\C(x))$ be the difference Galois group.   Since  $\Gal(\mathcal{Q} |\overline{C_{h}}(x))= \Gal(\mathcal{Q}_{\C} |\C(x))(\overline{C_{h}})$,  we find that $\Gal(\mathcal{Q}_{\C} |\C(x))$ is connected too. \par 
Let us prove  Theorem \ref{thm} by induction on $n$, the order of the equation.  The global  strategy of the proof will be quite similar to \cite{adamczewski2021hypertranscendence}, but the arguments will differ inside each step of the proof. 
\begin{enumerate}
\item Case $n=1$ (initialization of the induction).
\item Reduction to the affine case when $n\geq 2$.
\item Case of affine equations.
\end{enumerate}
\begin{center}
Step 1: Case $n=1$.
\end{center}
Let us prove Theorem \ref{thm} in the case $n=1$.
\begin{propo}\label{prop2}
Let $0\neq a\in \C(x)$.  If $f\in \mathcal{M}(\C)$ is solution of $\rho(y)=ay$ and is differentially algebraic over $\C(x)$, then $f\in  R_{ h, \exp}$.
\end{propo}

\begin{proof}
If $f=0$ the result is clear.  Assume that $f\neq 0$.   By Remark \ref{rem1},  we may apply 
 \cite[Corollary~3.4]{HS08}, to deduce that   there exist nonzero elements 
$c\in \C$ and $d(x) \in \C(x)$ such that $f(x+h)/f(x)=a(x)=c d(x+h)/d(x)$.   Let $c'\in\C$ such that $e^{c'h}=c$. We find that $c_1 (x):=\frac{f(x)}{d(x)e^{c'x}}$ is $\rho$-invariant.  We have proved that $c_1 (x)\in C_h$ and therefore
 $f (x)= c_1 (x)d(x)e^{c'x}\in R_{ h, \exp}$ as expected.  
\end{proof}

\begin{center}
Step 2: Reduction to the affine case.
\end{center}

Let us now begin the induction step of the proof of Theorem \ref{thm}.
Let us fix $n\geq 2$ and  assume that  Theorem \ref{thm} holds for equations of order strictly less than $n$.  We assume that  each  entry of $\mathcal{Y}$ is differentially algebraic  over $\C(x)$ and prove  that $\mathcal{Y}$ has entries in $R_{h,\exp}$.  By Proposition \ref{prop: colonne},  $\Gal(\mathcal{Q}_{\C} |\C(x))$ is reducible.   By \cite[Lemma 4.4]{adamczewski2021hypertranscendence}, there exists $T\in \mathrm{GL}_n (\C(x))$ such that $Z:=T\mathcal{Y}$ is solution of a bloc system $\rho (Z)=BZ$ where  $B=\left(\begin{array}{cc}
B_{1,1} &B_{1,2}   \\ 
0 & B_{2,2} 
\end{array} \right)$,  and $B_{i,i}$,  $i=1,2$,  has  size $n_i<n $. Let us further assume that  $n_1$ is minimal for this property. The goal of this step is to prove that when $n_1>1$ the result holds.  More precisely, we want to prove the following:

\begin{propo}\label{prop1}
If  each  entry of $\mathcal{Y}$ is differentially algebraic  over $\C(x)$ and $n_1>1$, then $\mathcal{Y}\in (R_{h,\exp})^{n}$.
\end{propo}

Let $G_1$ be the difference Galois group of $\rho (Y)=B_{1,1}Y$ over $\C(x)$.  The group $G_1$ is irreducible and connected,  see  \cite[Section 5.3]{adamczewski2021hypertranscendence}.   \par 
Let $(g_1, \dots , g_{n})^{\top}= T\mathcal{Y}$. The entries of $\mathcal{Y}$ are all differentially algebraic over $\C(x)$ if and only if the same holds for the entries of $(g_1, \dots , g_{n})^{\top}$. Similarly, $\mathcal{Y}\in (R_{h,\exp})^{n}$ if and only if  $(g_1, \dots , g_{n})^{\top}\in (R_{h,\exp})^{n}$. So it suffices to show that if all the entries of $(g_1, \dots , g_{n})^{\top}$ are differentially algebraic over $\C(x)$,  then $(g_1, \dots , g_{n})^{\top}\in (R_{h,\exp})^{n}$.  
Let $\mathcal{Y}_1$ be the vector $(g_1, \dots , g_{n_1})^{\top}$ and  $\mathcal{Y}_2$ be the vector $(g_{n_1 +1 }, \dots , g_{n})^{\top}$.  We have $\rho(\mathcal{Y}_1)=B_{1,1} \mathcal{Y}_1+B_{1,2}\mathcal{Y}_2$ and $\rho(\mathcal{Y}_2)=B_{2,2} \mathcal{Y}_2$.   By induction hypothesis, when each entry of $\mathcal{Y}$ is differentially algebraic,  we find that  $\mathcal{Y}_2\in (R_{h,\exp})^{n_2}$.  
Then, when each entry of $\mathcal{Y}$ is differentially algebraic over $\C(x)$, we have 
\begin{equation}\label{eq2}
\mathcal{Y}\in (R_{h,\exp})^{n}\Longleftrightarrow \mathcal{Y}_1\in (R_{h,\exp})^{n_1}.
\end{equation}
Toward the proof of Proposition \ref{prop1},  we are going to prove that when each  entry of $\mathcal{Y}$ is differentially algebraic  over $\C(x)$, then the entries of $\mathcal{Y}$ belong to a certain ring. 
  Since $\mathcal{Y}_2\in (R_{h,\exp})^{n_2}$, there exist $\lambda_1,\dots, \lambda_k \in \C$  and $\ell\in \N^*$ such that   the entries of  $ \mathcal{Y}_2$ belong to $C_{\ell h}(x)[e^{\lambda_1 x},\dots, e^{\lambda_k x}]$.  By replacing the $\rho$-equation with a $\rho^{\ell}$-equation,  we may reduce to the case where $\ell=1$. 
Recall that $\widetilde{C_h}$ is a differentially closed field containing $C_h$.  Let $\widetilde{K_0}$,  be the $(\rho,\delta)$-ring that is the total ring of fraction of the ring generated by the entries of $\mathcal{Y}_2$ over $\widetilde{C_h}(x)$. 

\begin{lem}\label{lem4} If  each  entry of $\mathcal{Y}$ is differentially algebraic  over $\C(x)$,  and $n_1>1$,  then $\mathcal{Y}\in (\widetilde{K_0})^{n}$. 
\end{lem}

\begin{proof}
Assume that each  entry of $\mathcal{Y}$ is differentially algebraic  over $\C(x)$. Then each entry of $\mathcal{Y}_1$ is  differentially algebraic  over $\C(x)$. Furthermore, since $(\mathcal{Y}_1,\mathcal{Y}_2)^{\top}=T\mathcal{Y}$, we have 
$\mathcal{Y}\in (\widetilde{K_0})^n$ if and only if $\mathcal{Y}_1\in (\widetilde{K_0})^{n_1}$.  Let us prove the contrapositive:  let us assume that  $ \mathcal{Y}_1\notin (\widetilde{K_0})^{n_1}$ and let us prove that $n_1=1$. Recall that $U\in \mathrm{GL}_n(\mathcal{M}(\C))$ is such that 
$\rho(U)=AU$. Let $V=TU$; this is a fundamental matrix of  $\rho(Y)=BY$.  By Lemma  \ref{lem6},  
 $\mathcal{Q}_S$, 
the total ring of fractions of $\widetilde{C_h}(x) \{ V,1/\det(V) \}_{\delta}$, is a $(\rho,\delta)$-Picard-Vessiot extension for $\rho(Y)=BY$ over $\widetilde{C_h}(x)$.
  Let $H=\mathrm{Gal}^{\delta}(\mathcal{Q}_S|\widetilde{C_h}(x))$ be the $(\rho,\delta)$-Galois group of 
$\rho (Y)=BY$ over $\widetilde{C_h}(x)$.  
By the $(\rho,\delta)$-Galois correspondence, see \cite[Theorem~2.7]{HS08},
we deduce the existence of some ${\sigma \in H}$, such that $\sigma(g_{i})=g_i$ for every $i$, $n_1+1\leq i\leq n$ and 
$$\left( \sigma(g_{i})\right)_{  i\leq n}^{\top}\neq\left(g_{i}\right)_{  i\leq n}^{\top}.
$$
Hence $w:=\mathcal{Y}_1-\sigma(\mathcal{Y}_1)$ is a nonzero vector. 
Since the coordinates of $\mathcal{Y}_1$ are differentially algebraic over $\widetilde{C_h}(x)$ and 
$\sigma$ belongs to the  $(\rho,\delta)$-Galois group $H$, 
the coordinates of $\sigma(\mathcal{Y}_1)$ are also differentially algebraic over $\widetilde{C_h}(x)$.  Therefore
the coordinates of $w$ are  differentially algebraic over $\widetilde{C_h}(x)$. 
Furthermore, $\mathcal{Y}_1$ and $\sigma(\mathcal{Y}_1)$ are both solution to the system 
$$
\rho(Y)=B_{1,1}Y+B_{1,2}\mathcal{Y}_2 \,.
$$
It follows that 
$$
\rho(w) =B_{1,1}w\,.
$$

Since 
we have 
$$
\widetilde{C_h}=\widetilde{C_h}(x)^\rho\subset (\mathcal{Q}_S)^\rho =\widetilde{C_h}\,,
$$ 
we find by an adaptation of Lemma \ref{lem2},   the existence of a positive integer $s$ and a  $(\rho^{s},\delta)$-Picard-Vessiot field extension 
$\mathcal{Q}_1$ for the system $\rho^s (Y)=(B_{1,1})_{[s]}Y$ 
over $\widetilde{C_h}(x)$ such that $w$ is the first column of a fundamental matrix $W$. 
Furthermore, 
the difference Galois group of   $\rho^s (Y)=(B_{1,1})_{[s]}Y$ over $\C(x)$ is equal to $G_1$ since the latter is connected.  
Let us build a fundamental matrix of $\rho^s (Y)=(B_{1,1})_{[s]}Y$ with coefficients in $\mathcal{M}(\C)$ and whose all coordinates of the first column are differentially algebraic over $\widetilde{C_h}(x)$.
 By \cite[Theorem 1]{Pra}, there exists $W_1\in \mathrm{GL}_{n_1}(\mathcal{M}(\C))$ such that $\rho^{s}(W_1)=(B_{1,1})_{[s]}W_1$.  By Lemma~\ref{lem6}, the total ring of fractions of $\widetilde{C_h}(x) \{ W_1,1/\det(W_1) \}_{\delta}$ is a $(\rho^s,\delta)$-Picard-Vessiot extension for $\rho^{s}(Y)=(B_{1,1})_{[s]}Y$ over $\widetilde{C_h}(x)$.  Recall that the entries of $w$ are differentially algebraic over $\widetilde{C_h}(x)$.
 Since the two  $(\rho^s,\delta)$-Picard-Vessiot ring extensions $\widetilde{C_h}(x) \lbrace W ,1/\det(W)\rbrace_{\delta}$ and $\widetilde{C_h}(x) \lbrace W_1, 1/\det(W_1) \rbrace_{\delta}$  are isomorphic, we deduce that without loss of generality,  we may assume that the entries of the first column of $W_1$ are differentially algebraic over $\widetilde{C_h}(x)$ too. 
The group $G_{1}$ being 
connected and irreducible, Proposition~\ref{prop: colonne} implies that $n_1=1$.  This proves the result by contraposition.
\end{proof}
 Let us descend from $\widetilde{K_0}$ to $\widetilde{C_h}\otimes R_{h,\exp}$.

 \begin{lem}\label{lem3} If  $\mathcal{Y}\in (\widetilde{K_0})^{n}$,  then $\mathcal{Y}\in (\widetilde{C_h}\otimes R_{h,\exp})^{n}$. 
\end{lem}

\begin{proof}
Recall that   $\mathcal{Y}$ is solution of  $\rho(Y)=AY$ and is the first column of the fundamental matrix $U\in \mathrm{GL}_n(\mathcal{M}(\C))$.
By Lemma \ref{lem5},  ${\mathcal{R}=\widetilde{C_h}(x)[U,1/\det(U)]}$ is a Picard-Vessiot ring extension for $\rho(Y)=AY$ over 
$\widetilde{C_h}(x)$.  Let us consider  $\lambda_1,\dots,\lambda_k \in \C$ used in the definition of  $\widetilde{K_0}$.
Let $W=\mathrm{Diag}(e^{\lambda_1 x},\dots, e^{\lambda_k x} )$.
 By Lemma~\ref{lem5},   $\mathcal{R}_1=\widetilde{C_h}(x)[W,1/\det(W)]$, is a  Picard-Vessiot ring extension for $\rho(Y)=\mathrm{Diag}(e^{\lambda_1 h},\dots, e^{\lambda_k h}) Y$ over 
$\widetilde{C_h}(x)$.
The following  idea  has been suggested by Charlotte Hardouin. 
Consider the ideal $$I_1=\{s\in \mathcal{R}_1 \vert s \mathcal{Y} \in \big(\mathcal{R}_1\big)^{n}\} .$$ Let us prove that it is a $\rho$-ideal. 
Let $s\in I_1$. 
We have $$\rho(s)\mathcal{Y}=\rho(s)A^{-1}A\mathcal{Y}=A^{-1}\rho(s\mathcal{Y}).$$
With  $s\mathcal{Y} \in \big(\mathcal{R}_1\big)^{n}$ we therefore deduce that we have   ${A^{-1}\rho(s \mathcal{Y}) \in \big(\mathcal{R}_1 \big)^{n}}$. 
Hence $\rho(s)\in I_1$.   This shows that $I_1$ is a $\rho$-ideal.  The latter is not reduced to $(0)$ since it contains the common denominator of the entries of $\mathcal{Y}$.
Since $\mathcal{R}_1$ is a simple $\rho$-ring,  and $I_1\neq (0)$, we deduce that $1\in I_1$, proving that $\mathcal{Y}$ has its entries  in  $\widetilde{C_h}(x)[e^{\lambda_1 x},\dots, e^{\lambda_k x},e^{-(\lambda_1 +\dots+ \lambda_k)x}]$.   This completes the proof.
\end{proof}

Let us finish the proof of Proposition \ref{prop1}.  Assume that the entries of $\mathcal{Y}$ are differentially algebraic over $\C(x)$ and $n_1>1$.   By  Lemma \ref{lem4},   and Lemma~\ref{lem3}, ${\mathcal{Y}\in (\widetilde{C_h}\otimes R_{h,\exp})^{n}}$.
The following lemma will terminate the proof of Proposition~\ref{prop1}.
\begin{lem}\label{lem1}
If $\mathcal{Y}\in (\widetilde{C_h}\otimes R_{h,\exp})^{n}$,  then  $\mathcal{Y}\in (R_{h,\exp})^{n}$. 
\end{lem}

\begin{proof}
Fix $1\leq i \leq n$ and let us prove that the entry number $i$ of $\mathcal{Y}$ belongs to $R_{h,\exp}$. Since the dependence in $i$ has no effect on the proof let us omit it in the notation and let $\widehat{y}$ be the entry number $i$ of $\mathcal{Y}$. Since $\widehat{y}\in \widetilde{C_h}\otimes R_{h,\exp}$, we have  $\widehat{y}=\tilde{c}_1 f_1 +\dots+ \tilde{c}_{\kappa}f_{\kappa}$ with $\tilde{c}_i\in \widetilde{C_h}$ and $f_i\in R_{h,\exp}$. Clearing denominators, we find the existence of $0\neq P\in C_h[x]$ of minimal degree such that 
we may write $P\widehat{y}=c_1\widehat{z}_1+\dots+c_{\kappa}\widehat{z}_{\kappa}$ with $c_i\in\widetilde{C_h}$ and  $\widehat{z}_i=\sum_j d_{i,j} e^{\lambda_{i,j} x}$, $d_{i,j} \in C_h [x]$, $\lambda_{i,j} \in \C$.
Since $\widehat{y}\in R_{h,\exp}$ if and only if $P\widehat{y}\in R_{h,\exp}$ it suffices to show that  $P\widehat{y}\in R_{h,\exp}$. Without loss of generality we may assume that the $\widehat{z}_i$ are $C_h$-linearly independent. Let us see  $\widehat{z}_i$ as a polynomial with coefficients in $C_h$, in $x$ and various exponentials  $e^{\lambda x}$, $\lambda\in \C$.
Then, 
$\widehat{z}_i$ is sum and product of elements of $C_h$, powers of $x$, and $e^{\lambda x}$, $\lambda\in \C$. For each of   those functions, there exists 
 a nontrivial linear difference equation in coefficients in $\C(x)$ that  vanishes on it. The set of  meromorphic functions on $\C$ annihilated by  some nontrivial linear difference equation in coefficients in $\C(x)$ form a ring. It then follows that for all $i$, there is a nontrivial linear difference equation with coefficients in $\C(x)$ annihilating $\widehat{z}_i$. Furthermore, we may  consider $\mathcal{L}$ a nonzero linear  difference equation in coefficients in $\C(x)$ of order greater than $\kappa$, such that for all $i$, $\mathcal{L}(\widehat{z}_i)=0$. Let  $\rho (Y)=\widetilde{A}Y$ be the corresponding system, where $\widetilde{A}$ is a companion matrix, and let $m\geq \kappa$ be its size. Since for all $i$, $\rho(c_i)=c_i$, we have $\mathcal{L}(P\widehat{y})=0$. 
With  Lemma~\ref{lem2} we obtain the existence of $W\in \mathrm{GL}_m (\mathcal{M}(\C))$ with $\rho(W)=\widetilde{A}W$ such that the $\kappa$ first columns are the $(\widehat{z}_i, \rho(\widehat{z}_i),\dots, \rho^{m-1}(\widehat{z}_i))^{\top}$, $1\leq i \leq \kappa$. Similarly, let $U\in \mathrm{GL}_m (\mathcal{M}(\C))$ with $\rho(U)=\widetilde{A}U$ such that the  first
column is  $(P\widehat{y}, \rho(P\widehat{y}),\dots, \rho^{m-1}(P\widehat{y}))^{\top}$.
Since $U$ and $W$ are solutions of $\rho(Y)=\widetilde{A}Y$ with  coefficients in $\mathrm{GL}_m (\mathcal{M}(\C))$, there exists $C_0\in\mathrm{GL}_m ((\mathcal{M}(\C))^{\rho})=\mathrm{GL}_n (C_h)$ such that $U=WC_0$. Recall that $P\widehat{y}=c_1\widehat{z}_1+\dots+c_{\kappa}\widehat{z}_{\kappa}$ with $c_i\in\widetilde{C_h}$ so that for all $i$, 
$\rho^i (P\widehat{y})=c_1\rho^i(\widehat{z}_1)+\dots+c_{\kappa}\rho^i(\widehat{z}_{\kappa})$.
Let us replace the first column of $C_0$ by  $(c_1,\dots ,\dots,c_{\kappa},0\dots)^{\top}$ to obtain a new matrix $\widetilde{C}_0$ with coefficients in $\widetilde{C_h}$. In the ring $\widetilde{C_h}\otimes \mathcal{M}(\C)$ we have
 $U=W\widetilde{C}_0$ and then multiplying to the left by $W^{-1}$, we obtain $W^{-1}U=C_0=\widetilde{C}_0$. Since the first column of $C_0$ has coefficients in $C_h$ we find that the $c_i$ are in $C_h$. Hence $P\widehat{y}\in R_{h,\exp}$ as expected.
\end{proof}

\begin{center}
Step 3: Affine case.
\end{center}
Let us finish the proof of Theorem \ref{thm}.  Let us recall that we assume that the entries of $\mathcal{Y}$ are differentially algebraic over $\C(x)$ and we have to prove that $\mathcal{Y}\in (R_{h,\exp})^{n}$.
By  Proposition~\ref{prop1},  it remains to consider the case where $n_1=1$. By \eqref{eq2} it suffices to show that $g_1\in R_{h,\exp}$.
Recall that $g_1$ is a differentially algebraic solution of the affine equation $\rho(g_1)=ag_1+b$ where $a=B_{1,1}\in \C(x)$ and $b=B_{1,2}(g_2,\dots, g_n)^\top \in R_{h,\exp}$.   
Let $\widetilde{C_h}$ be a differentially closed field containing   $C_h$.  By \cite[Proposition~3.8]{HS08},     the $(\rho,\delta)$-Galois group of $\rho(y)=ay$ over $\widetilde{C_h}(x)$ is not all of $\mathrm{GL}_1 (\widetilde{C_h})$.    By \cite[Theorem~1]{Pra},  let $h_0\in \mathcal{M}(\C)$ be a nonzero solution of $\rho(y)=ay$.   In virtue of  Lemma~\ref{lem6}, the total ring of fractions of $\widetilde{C_h}(x) \{ h_0,1/h_0 \}_{\delta}$ is a $(\rho,\delta)$-Picard-Vessiot extension for $\rho(Y)=aY$ over $\widetilde{C_h}(x)$.
With \cite[Proposition~6.26]{HS08},  $h_0$ is differentially algebraic over $\widetilde{C_h}(x)$. 
By Remark \ref{rem1}, we may  use    \cite[Corollary~3.4]{HS08},  to find that  there exist nonzero $c\in \C$ and $g\in \C(x)$ such that $a=c\rho(g)/g$.  Let $c'\in \C$ such that $e^{c' h}=c$. We have $g_1\in R_{h,\exp}$ if and only if $\hat{g}:=\frac{g_1}{e^{c'x}g}\in R_{h,\exp}$.   We have $\rho(\hat{g})=\hat{g}+\frac{b}{\rho(e^{c'x}g)}$.
 The latter is solution of an equation of the form $\rho(y)=y+b'$ where $b':=\frac{b}{\rho(e^{c'x}g)}\in R_{h,\exp}$.
 So without loss of generality  we may reduce to the case where $a=1$.  Then, $g_1$ is solution of $\rho(y)=y+b$ where  $b\in R_{h,\exp}$. 
Let $\ell\in \N^*$ and $\lambda_1,\dots, \lambda_k  \in \C$,    such that  $b\in C_{\ell h} (x)[e^{\lambda_1 x},\dots, e^{\lambda_k x}]$.
We have $\rho^{\ell}(g_1)=g_1+b_{[\ell]}$, with $b_{[\ell]}:=b+\dots+\rho^{\ell-1}(b)\in C_{\ell h} (x)[e^{\lambda_1 x},\dots, e^{\lambda_k x}]$.  Since $R_{\ell h, \exp} \subset R_{h, \exp}$ is   $(\rho^{\ell},\delta)$-ring,  we may  replace the $\rho$-equation by a $\rho^{\ell}$-equation to reduce to the case $\ell=1$.  The fundamental matrix $\mathrm{Diag}(e^{\lambda_1 x},\dots, e^{\lambda_k x})$ is solution of $\rho(Y)=\mathrm{Diag}(e^{\lambda_1 h},\dots, e^{\lambda_k h})Y$ so we have $k$ vector columns that are linearly independent. We deduce with  Lemma \ref{lem2},  with $\kappa=k$, that there exists $\ell\geq 1$, such that 
$\mathrm{Diag}(e^{\lambda_1 x},\dots, e^{\lambda_k x})$ is a fundamental matrix of $\rho^{\ell}(Y)=\mathrm{Diag}(e^{\lambda_1 h},\dots, e^{\lambda_k h})_{[\ell]}Y$ and 
the $(\rho^{\ell},\delta)$-Picard-Vessiot ring extension for  $\rho^{\ell}(Y)=\mathrm{Diag}(e^{\lambda_1 h},\dots, e^{\lambda_k h})_{[\ell]}Y$ over $\widetilde{C_{h}}(x)$  is an integral domain. Then 
$\widetilde{C_{h}}(x)\{e^{\lambda_1 x},\dots,e^{\lambda_k x},e^{-(\lambda_1+\dots+ \lambda_k) x}\}_{\delta}$ is a $(\rho^{\ell},\delta)$-Picard-Vessiot ring extension and is an integral domain. Again, replacing the $\rho$-equation by a $\rho^{\ell}$-equation we may reduce to the case where the $(\rho,\delta)$-ring $\widetilde{C_{h}}(x)\{e^{\lambda_1 x},\dots,e^{\lambda_k x},e^{-(\lambda_1+\dots+ \lambda_k) x}\}_{\delta}$ in an integral domain and its field of fractions has $\widetilde{C_{h}}$ as field of constants.

By   \cite[Proposition~3.1]{HS08},  there exists a linear differential operator with coefficients in  $\widetilde{C_{h}}$ and $g\in \mathrm{Frac}( \widetilde{C_{h}}(x)\{e^{\lambda_1 x},\dots, e^{\lambda_k x}\}_{\delta})= \widetilde{C_{h}}(x)(e^{\lambda_1 x},\dots, e^{\lambda_k x})$,  such that $L(b)=\rho(g)-g$.  
 Consider $W=\mathrm{Diag}(e^{\lambda_1 x},\dots,e^{\lambda_k x} )$.
The field $\widetilde{C_h}$ is algebraically closed.  By  Lemma~\ref{lem5},   $\mathcal{R}_1=\widetilde{C_h}(x)[e^{\lambda_1 x},\dots,e^{\lambda_k x},e^{-(\lambda_1+\dots+ \lambda_k) x}]$, is a  Picard-Vessiot ring extension for $\rho(Y)=\mathrm{Diag}(e^{\lambda_1 h},\dots, e^{\lambda_k h}) Y$ over 
$\widetilde{C_h}(x)$.
Consider the ideal  $$J_1=\{s\in \mathcal{R}_1 \vert s g \in \mathcal{R}_1\}.$$
Let us prove that it is a $\rho$-ideal. Let $s\in J_1$. We have $\rho(sg)\in \mathcal{R}_1$. But $\rho(sg)=\rho(s)(g+L(b))$ and $\rho(s)L(b)\in \mathcal{R}_1$, proving that $\rho(s)g\in \mathcal{R}_1$. Hence $\rho(s)\in J_1$ showing that $J_1$ is a $\rho$-ideal.
Since $\mathcal{R}_1$ is a simple $\rho$-ring,  and by construction $J_1\neq (0)$, we deduce that  $1\in J_1$ and therefore $$g\in \widetilde{C_{h}}(x)[e^{\lambda_1 x},\dots, e^{\lambda_k x},e^{-(\lambda_1 +\dots+ \lambda_k)x}].$$  
 Let us write $b=\sum b_i e^{\lambda_i x}$,  $g=\sum \tilde{g}_i e^{\lambda_i x} $,  $b_i,\tilde{g}_i\in  \widetilde{C_{h}}(x)$,  $\lambda_i\in \C$.
 Then,  using the $\widetilde{C_{h}}(x)$-linear independence of the $e^{\lambda x}$, $\lambda\in \C$,
  we find  equations of the form $L_i (b_i)=e^{h\lambda_i} \rho (\tilde{g}_i)-\tilde{g}_i $, where $L_i$ is a linear differential operator with coefficients in $\widetilde{C_{h}}$.
By Remark \ref{rem1}, we may  use    \cite[Lemma~6.4]{HS08},  to deduce the existence of $\hat{g}_i\in  \widetilde{C_{h}}(x)$ such that 
$b_i= e^{h\lambda_i}  \rho (\hat{g}_i)-\hat{g}_i $.   
Then,  there exists $g'\in  \mathrm{Vect}_{\widetilde{C_{h}}(x)}[e^{\lambda_1 x},\dots,  e^{\lambda_k x}]$, such that  $b=\rho(g')-g'$.  Equating the constant coefficients yields  the existence of $g''\in R_{h,\exp}$ such that 
$b=\rho(g'')-g''$.
Recall that $b=\rho(g_1)-g_1$.
We then find that $g_1-g''$ is $\rho$-invariant.  
Hence,  $g_1-g''\in C_{ h}$ and therefore $g_1\in  R_{h,\exp}$. This completes the proof. 

\bibliographystyle{alpha}
\bibliography{biblio}
\end{document}